\documentclass{aptpub}
\usepackage{amsmath,amssymb, mathrsfs}
\usepackage{graphicx}	
\usepackage{subfigure}
\usepackage{float}
\usepackage{fullpage}
\usepackage{listings}
\usepackage{amsfonts, dsfont}
\usepackage{enumerate}
\usepackage{xcolor}
\usepackage{graphicx}
\usepackage[mathscr]{euscript}
\authornames{Rachel Traylor} 
\shorttitle{Short title} 

\numberwithin{equation}{section}  

\newcommand{\mfH}{\mathfrak{H}}   \newcommand{\mfh}{\mathfrak{h}}

\newcommand{\mfR}{\mathfrak{R}}     \newcommand{\mfr}{\mathfrak{r}}

\newcommand{\mfV}{\mathfrak{V}}     \newcommand{\mfv}{\mathfrak{v}}

\newcommand{\mcB}{\mathcal{B}}

\newcommand{\mcH}{\mathcal{H}}

\newcommand{\mcW}{\mathcal{W}}    \newcommand{\mcw}{\mathcal{w}}

 \newcommand{\frall}{\; \forall \;}

\renewcommand{\epsilon}{\varepsilon}

\begin{document}

\title{A Stochastic Reliability Model of a Server under a Random Workload} 

\authorone[EMC, University of Texas at Arlington]{Rachel Traylor} 

\addressone{2421 Mission College Blvd. Santa Clara, CA 95050} 

\begin{abstract}
Traffic to any server is rarely constant over time. In addition, the workload brought by each service request is typically unknown in advance, and each request may bring a different workload to the server. Cha and Lee (2011) proposed a reliability model where each request brings an identical and constant workload. In this paper, we generalize the model to allow for requests to bring an unknown random stress to the server. Jobs arrive to the server via a nonhomogeneous Poisson process. Service times are random and i.i.d. Each job adds a random stress $\mcH_{j} \sim \mcH$ to the breakdown rate of the server until the job is completed. The survival function of such a server and the efficiency of the server are derived.  
\end{abstract}



\section{Introduction} 


A stochastic model for server reliability under nonconstant traffic has many applications including retail, logistics, manufacturing, and IT.   In \cite{chalee}, the authors considered a Web server wherein each request arrives via a nonhomogeneous Poisson process $\{N(t): t \geq 0\}$ with intensity $\lambda(t)$ and adds a constant stress $\eta$ as long as the job is in the system.  Assuming the baseline breakdown rate is $r_{0}(t)$, the breakdown rate process is defined as

\[\mcB(t) = r_{0}(t) + \eta\sum_{j=1}^{N(t)}\mathbb{I}(T_{j}\leq t \leq T_{j}+W_{j})\]

 where $N(t)$ is the nonhomogeneous Poisson process describing the arrivals of requests, $\{T_{j}\}_{j=1}^{N(t)}$ are the arrival times of requests, $\{W_{j}\}_{j=1}^{N(t)}$ are the service times, and $eta$ is the constant stress breakdown rate (CSBR) increase brought by each job. It is assumed that the $\{T_{j}\}_{j=1}^{N(t)}$ are independent and the $\{W_{j}\}_{j=1}^{N(t)}$ are i.i.d. with pdf $g_{W}(w)$ and distribution $G_{W}(w)$. Cha and Lee (2011) derive the survival function for such a server. \par 	
 The assumption that each job brings the same stress to a server, especially a Web server, does not reflect applications accurately. In many situations, each job will bring a different, unknown stress to the server, in addition to a random arrival and service time. Therefore, the following generalization and expansion of the model is proposed. It should be noted that the model is general enough to apply to any situation with a server or system that experiences nonconstant traffic with unknown and random workloads, or stresses, brought by each service request. \par 
 
\section{Random Stress Breakdown Rate}
Retaining the previous assumptions from [1] except for the supposition that every request adds a constant stress to the server, the following generalization is proposed.   Assume that each job $j$ adds a random stress $\mcH_{j}$ to the server for the duration of its time in the system. Suppose the stresses $\{\mcH_{j}\}_{j=1}^{N(t)} \overset{i.i.d}\sim \mcH$, where WLOG $\mcH$ is a discrete random variable with a finite sample space $S = \{\eta_{1},...,\eta_{m}\}$ and  probability distribution given by $P(\mcH = \eta_{i}) = p_{i}, i = 1,...,m, \sum_{i=1}^{m}p_{i}=1$. 
(Note, the proof is unaffected by the discreteness of $\mcH$. One may also take to be a continuous random variable.) Then the new breakdown rate process is given by 

\begin{equation}
\mcB(t) = r_{0}(t) + \sum_{j=1}^{N(t)}\mcH_{j}\mathbb{I}(T_{j} \leq t \leq W_{j})
\label{breakdown rate process}
\end{equation}
	 	
Let $Y$ be the random time to failure of the Web server under the workload from client requests. Then the survival function of the server, conditioned on arrival process of client requests ($N(t)$), arrival times 
($\{T_{j}\}_{j=1}^{N(t)}$), service times ($\{W_{j}\}_{j=1}^{N(t)}$), and stresses 
($\{\mcH_{j}\}_{j=1}^{N(t)}$) is given by 

\begin{align}
P\left(Y>t \left| N(t), \{T_{j}\}_{j=1}^{N(t)},\{W_{j}\}_{j=1}^{N(t)},\{\mcH_{j}\}_{j=1}^{N(t)}\right.\right) &= \exp\left(-\int_{0}^{t}\mcB(t)dt\right) \nonumber \\
	&= \bar{F}_{0}(t)\exp\left(-\sum_{j=1}^{N(t)}\mcH_{j}\min(W_{j}, t-T_{j})\right)
\label{conditional survival function}
\end{align}
	 		 	
where $\bar{F}_{0}(t) = \exp\left(-\int_{0}^{t}r_{0}(x)dx\right)$ . 

\section{Survival Function to RSBR Server}
The unconditional survival function $S_{Y}(t)$ of the Web server under the random stress breakdown rate (RSBR) model is given in the following theorem: 

\begin{thm}[RSBR Survival Function]
Suppose that jobs arrive to a server according to a nonhomogeneous Poisson process $\{N(t), t \geq 0\}$ with intensity function $\lambda(t)$ and denote $m(t) \equiv E[N(t)] = \int_{0}^{t}\lambda(s)ds$, where $m(t)$
  has an inverse. Let the arrival times $\{T_{j}\}_{j=1}^{N(t)}$ be independent. Let the service times $\{W_{j}\}_{j=1}^{N(t)}$ be i.i.d. with pdf $g_{W}(w)$ and distribution $G_{W}(w)$, and let them be mutually independent of all arrival times. Let the stresses $\{\mcH_{j}\}_{j=1}^{N(t)} \overset{i.i.d}\sim \mcH$, where the distribution of $\mcH$ is given above WLOG. Then
  \begin{align}
  S_{Y}(t) &= \bar{F}_{0}(t)\exp\left(-E_{\mcH}\left[\mcH\int_{0}^{t}e^{-\mcH w}m(t-w)\bar{G}_{W}(w)dw\right]\right)
  \label{unconditional survival function}
  \end{align}
 \label{unconditional survival function theorem}
 \end{thm}    
 
 \begin{proof}
 
\begin{align}
S_{Y}(t) &= P(Y > t) \nonumber \\
			&= E\left[P\left(Y > t| N(t), \{T_{j}\}_{j=1}^{N(t)},, \{W_{j}\}_{j=1}^{N(t)}, \{\mcH_{j}\}_{j=1}^{N(t)}\right)\right] \nonumber
\end{align}

Using the Law of Total Expectation,

\[E\left[\exp\left(-\sum_{j=1}^{N(t)}\mcH_{j}\min(W_{j}, t- T_{j})\right)\right] = E\left[\left.E\left[\exp\left(-\sum_{j=1}^{N(t)}\mcH_{j}\min(W_{j}, t- T_{j})\right)\right] \right|N(t), \{\mcH_{j}\}_{j=1}^{N(t)}\right]\]

Let $N(t) = n$ and $\mcH_{j} = \eta_{i_{j}}$ for some $i \in \{1,...,m\}, j = 1,...,n$. Since the stresses brought by each job are mutually independent of the arrival times, $f_{T_{1},...,T_{n}|N(t),\mcH_{1},...,\mcH_{n}}(t_{1},...,t_{n}|n,\eta_{i_{1}},...,\eta_{i_{n}}) = f_{T_{1},...,T_{n}|N(t)}(t_{1},...,t_{n}|n)$. Let $T'_{1},...,T'_{n}$ be i.i.d. random variables with pdf $f(x) = \frac{\lambda(x)}{m(t)}$. Then by Lemma~\ref{order_stat}, the conditional join distribution of $T_{1},...,T_{n}$ is equal to the distribution of $T'_{[1]},...,T'_{[n]}$, where $\{T'_{[i]}\}_{i=1}^{n}$ are the order statistics of $T'_{1},...,T'_{n}$ and is given by 
\[f_{T_{1},...,T_{n}|N(t)}(t_{1},...,t_{n}|n) = n!\prod_{j=1}^{n}\frac{\lambda(t_{j})}{m(t)}, \quad 0 \leq t_{1} \leq \ldots \leq t_{n} \leq t\]

Then

\begin{align}
 E\left.\left[\exp\left(-\sum_{j=1}^{N(t)}\mcH_{j}\min(W_{j}, t-T_{j})\right) \right| N(t), \mcH_{1},...\mcH_{N(t)}\right] 
 	&=E\left[\exp\left(-\sum_{j=1}^{n}\eta_{i_{j}}\min(W_{j}, t-T_{j})\right)\right] \nonumber \\
 	&= E\left[\exp\left(-\sum_{j=1}^{n}\eta_{i_{[j]'}}\min(W_{j}, t-T_{(j)})\right)\right] \nonumber \\
 	&= E\left[\exp\left(-\sum_{j=1}^{n}\eta_{i_{j'}}\min(W_{j}, t-T_{j'})\right)\right]	\nonumber \\
 	&= E\left[\prod_{j=1}^{n}\exp\left(-\eta_{i_{j'}}\min(W_{j}, t-T_{j'})\right)\right]\nonumber \\
 	&= \prod_{j=1}^{n}E\left[\exp\left(-\eta_{i_{j'}}\min(W_{j}, t-T_{j'})\right)\right] \nonumber \\
 \end{align}
 
 Where the equalities hold because $\{W_{j}\}_{1}^{n}, \{T_{j'}\}_{1}^{n}$ are i.i.d., respectively, and mutually independent. Now, fix $j'$. Then $\eta_{i_{j'}}$ is also fixed, and 
 
 \begin{align}
 E\left[\exp\left(-\eta_{i_{j'}}\min(W_{j}, t-T_{j'})\right)\right] &= E\left[\left.E\left[-\eta_{i_{j'}}\min(W_{j}, t-T_{j'})\right| W_{j}\right]\right] \nonumber \\
 	&= \frac{1}{m(t)}\left(m(t) - \eta_{i_{j'}}\int_{0}^{t}\exp(-\eta_{i_{j'}}w)m(t-w)\bar{G}_{W}(w)dw\right) 
 \label{thingy}
 \end{align}
 by~\ref{cond_on_Wj}. But  Equation \eqref{thingy} is true $\frall j$, so
  
  \begin{equation}
  E\left[\left.\exp\left(-\sum_{j=1}^{N(t)}\mcH_{j}\min(W_{j}, t-T_{j})\right) \right| N(t), \mcH_{1},...\mcH_{N(t)}\right] 
  	= \prod_{j=1}^{n}\frac{1}{m(t)}\left(m(t) - \eta_{j'}^{i}\int_{0}^{t}e^{-\eta_{i_{j'}}w}m(t-w)\bar{G}_{W}(w)dw\right) 
  \label{big_thingy}
  \end{equation}
  
  Now it remains to take the expectation of \eqref{big_thingy} over $N(t), \mcH_{1},..,\mcH_{N(t)}$. Let $h_{j}^{i_{j}}(t) = \left(m(t) - \eta_{j}^{i}\int_{0}^{t}\exp(-\eta_{j}^{i}w)m(t-w)\bar{G}_{W}(w)dw\right)$. Then
  
   \[E\left[\prod_{j=1}^{n}\frac{1}{m(t)}\left(m(t) - \eta_{i_{j}}\int_{0}^{t}e^{-\eta_{i_{j}}w}m(t-w)\bar{G}_{W}(w)dw\right)\right] =  E\left[\prod_{j=1}^{n}\frac{1}{m(t)}h_{j}^{i}(t)\right] \]
    and hence 
    
    \begin{align}
    E\left[\prod_{j=1}^{n}\frac{1}{m(t)}h_{i_{j}}(t)\right] 
    		&= \sum_{n=0}^{\infty}\sum_{i_{n} = 1}^{m} \cdots\sum_{i_{2} = 1}^{m}\sum_{i_{1} = 1}^{m}\left(\prod_{j=1}^{n}\frac{1}{m(t)}h_{i_{j}}(t)\right)P(\mcH_{1} = \eta_{i_{1}})P(\mcH_{2} = \eta_{i_{2}})\cdots P(\mcH_{n} = \eta_{i_{n}})P(N(t) = n)\nonumber \\
    		&= \sum_{n=0}^{\infty}\frac{1}{m(t)^{n}}\left(E_{\mcH_{1}}[h_{i_{1}}(t)]\right)\sum_{i_{n}=1}^{m}\cdots\sum_{i_{2}=1}^{m}\left(\prod_{j=2}^{n}h_{i_{j}(t)}\right)P(\mcH_{2} = \eta_{i_{2}})\cdots P(\mcH_{n} = \eta_{i_{n}})P(N(t)=n) \nonumber \\
    		&= \sum_{n=0}^{\infty}\frac{1}{m(t)^{n}}\left(\prod_{j=1}^{n}E_{\mcH_{j}}[h_{i_{j}}(t)]\right) P(N(t)=n) \nonumber \\
    		&=  \sum_{n=0}^{\infty}\frac{1}{m(t)^{n}}\left(\prod_{j=1}^{n}E_{\mcH_{j}}[h_{i_{j}}(t)]\right)\frac{m(t)^{n}}{n!}e^{-m(t)} \nonumber 
   \end{align}
 Since $\mcH_{j} \overset{i.i.d.}\sim \mcH \frall j$,
  \begin{align}
  \sum_{n=0}^{\infty}\prod_{j=1}^{n}E_{\mcH_{j}}[h_{i_{j}}(t)]\frac{e^{-m(t)}}{n!} &= \sum_{n=0}^{\infty}\frac{e^{-m(t)}}{n!}\left(E_{\mcH}[h_{j}(t)]\right)^{n} \nonumber \\
  		&= \exp\left(-E_{\mcH}\left[\mcH\int_{0}^{t}e^{-\mcH w}m(t-w)\bar{G}_{W}(w)dw\right]\right) \nonumber
  \end{align}
  
  The hazard function $r(t)$ is then given by 
  \begin{equation}
  r(t) = -\frac{d\ln(S_{Y}(t))}{dt} = r_{0}(t) + E_{\mcH\left[\mcH\int_{0}^{t}e^{-\mcH w}m(t-w)\bar{G}_{W}(w)dw\right]}
  \end{equation}
 \end{proof}
 
 \section{Efficiency Measure of the Server under RSBR model}
 
 If the server crashes during operation, then it is rebooted. We assume that after rebooting, the server is "as good as new", i.e. upon fixing malfunctions, the performance can be regarded as the same as prior to the crash. Cha and Lee (2011) assumed that the arrival process of the client requests after rebooting "restarts". Formally, that 
 \begin{enumerate}[(1)]
 \item The arrival process after rebooting,$\{N_{rb}(t), t \geq 0\}$, remains a nonhomogeneous Poisson process with the same intensity function $\lambda(t)$ as before, and 
 \item $\{N_{rb}(t), t \geq 0\}$ is independent of $\{N(t), t \geq 0\}$.
 \end{enumerate}

 Hence, after a reboot, we may refer to $\{N_{rb}(t), t \geq 0\}$ as $\{N(t), t \geq 0\}$. In addition, assume the time to reboot of the server follows any continuous distribution $H(t)$ with a mean $\nu$, as in \cite{chalee}. \par 
  	Let $M(t)$ be the total number of jobs completed by the server during $(0,t]$. We wish to measure the performance of the server at any time t, and so retain the definition given in \cite{chalee}.
 
 \begin{defnn}[Efficiency of the Server]
 The \textit{efficiency} is the server, $\psi$, is defined as 
 \[\psi := \lim\limits_{t \to \infty}\frac{E[M(t)]}{t}\]
 \end{defnn}
 Standard renewal theory gives $\psi = \frac{E[M]}{E[Y]+\nu}$ where $M$ is the number of jobs completed in a renewal cycle, $Y$ is the length of a renewal cycle, and $\nu$ is the mean time to reboot \cite{ross}. Under the RSBR model, $\psi$ has a closed form given in the following theorem.
 
 \begin{thm}[Efficiency of a Server under RSBR]
 Let the mean time to reboot of the server be $\nu$. Under the same conditions as Theorem~\ref{unconditional survival function theorem}, 
 \begin{equation}
 \psi  = \frac{1}{\int_{0}^{\infty}S_{Y}(t)dt + \nu}\left\{\int_{0}^{\infty}\exp\left(-\int_{0}^{t}r_{0}(x)-\int_{0}^{t}\lambda(x)dx+E_{\mcH}[a(t)+b(t)]\right)\left[r_{0}(t)E_{\mcH}[a(t)] + E_{\mcH}[\mcH a(t)b(t)]\right]dt \right\}
 \label{efficiency_formula}
 \end{equation}
 where $a(t) := \int_{0}^{t}e^{-\mcH v}g_{W}(v)m(t-v)dv$ and $b(t) := \int_{0}^{t}e^{-\mcH(t-r)}\bar{G}_{W}(t-r)\lambda(r)dr$.
\label{efficiency theorem}
 \end{thm}
 
 \begin{proof}
 $E[Y] = \int_{0}^{\infty}S_{Y}(t)dt$, where $S_{Y}(t)$ is given in \eqref{unconditional survival function}. Hence it suffices to derive $E[M]$. The total number of jobs completed in a renewal cycle is given by $M = \sum_{j=1}^{N(Y)}\mathds{1}_{T_{j} + W_{j} \leq Y}$. $N(Y)$ is finite for any selected renewal cycle. Hence, letting $\{(R_{j}, V_{j})\}_{j=1}^{N(Y)}$ be a random permutation of $\{(T_{j}, W_{j})\}_{j=1}^{N(Y)}$, $M = \sum_{j=1}^{N(Y)}\mathbb{I}_{R_{j}+V_{j} \leq Y}$ due to the mutual independence of arrival times and service times. For convenience and clarity, we introduce the following notation: $\mathfrak{R} = \{R_{1},...,R_{n}\}$, 
  	 $\mathfrak{V} = \{V_{1},...,V_{n}\}$, 
   	 $\mathfrak{H} = \{\mcH_{1},..., \mcH_{n}\}$ 
    with observed values
   	$\mfr = \{r_{1},...,r_{n}\}$, 
   	$\mfv = \{v_{1},...,v_{n}\}$,  \\
   	$\mfh = \{\eta_{i_{1}},...,\eta_{i_{n}}\}$. By Bayes' Theorem,  
   	\[f_{\mfR, \mfV, \mfH, Y, N}(\mfr, \mfv, \mfh, t,n) = 
   	     f_{Y|\mfR, \mfV, \mfH, N}(t| \mfr, \mfv, \mfh, n)f_{\mfR, \mfV, \mfH, N}(\mfr, \mfv, \mfh, n)\]
   	 By~\ref{cond_dist_for_efficiency}, the conditional distribution $f_{Y|\mfR, \mfV, \mfH, N}(t| \mfr, \mfv, \mfh, n)$ is given by 
   	 \begin{equation}
   	  	f_{Y|\mfR, \mfV, \mfH, N}(t| \mfr, \mfv, \mfh, n) = \exp \left(-\int_{0}^{t}r_{0}(s)ds - \sum_{j=1}^{n}\eta_{j}^{i_{j}}\min(v_{j}, t-r_{j})\right) \left(r_{0}(t) + \sum_{j=1}^{n}\eta_{j}^{i_{j}}\mathds{1}_{v_{j} > t-r_{j}}\right)
   	  \label{f Y|R,V,N,H}
   	  \end{equation}
   	  
   	  Since $\mcH_{j}, j = 1,...,n$ are i.i.d. and mutually independent of $\mfR, \mfV,$ and $N$, \[f_{\mfR, \mfV, \mfH, N}(\mfr, \mfv, \mfh, n) = f_{\mfR, \mfV, N}(\mfr, \mfv, n)f_{\mfH}(\mfh) = f_{\mfR, \mfV, N}(\mfr, \mfv, n)\prod_{j=1}^{n}P(\mcH_{j} = \eta_{i_{j}})\] By~\ref{f(r,v,n)},
   	  \[f_{\mfR, \mfV, N}(\mfr, \mfv, n) =     \frac{1}{n!}\prod_{j=1}^{n}\exp\left(\int_{0}^{t}\lambda(x)dx\right)\lambda(r_{j})g_{W}(v_{j})\]
   	  and thus 
   	  \begin{equation}
   	  f_{\mfR,\mfV,\mfH,N}(\mfr,\mfv, \mfh, n) = \frac{1}{n!}\prod_{j=1}^{n}\exp\left(-\int_{0}^{t}\lambda(s)ds\right)\lambda(r_{j})g_{W}(v_{j})P(\mcH_{j} = \eta_{i_{j}})
   	  \label{f_{R,V,H,N}}
   	  \end{equation}
   	  
   	  Finally, by multiplying ~\eqref{f Y|R,V,N,H} and \eqref{f_{R,V,H,N}}, 
   	  \begin{align}
   	   f_{\mfR, \mfV, \mfH, Y, N(t)}&(\mfr, \mfv, \mfh, t, n)  \nonumber \\
   	   			&= \frac{1}{n!}\left(r_{0}(t) + \sum_{j=1}^{n}\eta_{i_{j}}\mathds{1}_{v_{j} > t-r_{j}}\right)\exp\left(-\int_{0}^{t}r_{0}(x)dx - \sum_{j=1}^{n}\eta_{i_{j}}\min(v_{j}, t-r_{j})\right) \nonumber \\
   	   			&\quad\times\prod_{j=1}^{n}\exp\left(-\int_{0}^{t}\lambda(x)dx\right)\lambda(r_{j})g_{W}(v_{j})P(\mcH_{j} = \eta_{i_{j}}) \nonumber \\
   	   			&= \frac{1}{n!}r_{0}(t)\exp\left(-\int_{0}^{t}r_{0}(x)dx - \int_{0}^{t}\lambda(x)dx\right)\nonumber \\
   	   				&\quad\times \prod_{j=1}^{n}\exp(-\eta_{i_{j}}\min(v_{j}, t-r_{j}))\lambda(r_{j})g_{W}(v_{j})P(\mcH_{j} = \eta_{i_{j}}) \nonumber \\
   	   				&\quad+ \frac{1}{n!}\exp\left(-\int_{0}^{t}r_{0}(x)dx - \int_{0}^{t}\lambda(x)dx\right)\left(\sum_{j=1}^{n}\eta_{i_{j}}\mathds{1}_{v_{j} > t-r_{j}}\right) \nonumber \\
   	   				&\quad\times \prod_{j=1}^{n}\exp(-\eta_{i_{j}}\min(v_{j}, t-r_{j}))\lambda(r_{j})g_{W}(v_{j})P(\mcH_{j} = \eta_{i_{j}}) \nonumber 
   	   \label{f r,v,h,y,n}
   	   \end{align}
   	    Then 
   	    \begin{align}
   	    E[M] &= E[\sum_{j=1}^{N(Y)}\mathds{1}(R_{j} + V_{j} \leq Y)] \nonumber \\
   	    		&= \sum_{n=1}^{\infty}\int_{0}^{t}\left[\sum_{j=1}^{n}\sum_{i_{j} = 1}^{m} \int_{0}^{t}\int_{0}^{t-r_{j}}\int_{0}^{t}\int_{0}^{\infty}f_{\mfR, \mfV, \mfH, Y, N(t)}(\mfr, \mfv, \mfh, t,n)d\mfv_{-j}d\mfr_{-j}dv_{j}dr_{j}\right]dt \nonumber \\
   	   		&= \sum_{n=1}^{\infty}\int_{0}^{\infty}\frac{1}{n!}r_{0}(t)\exp\left(-\int_{0}^{t}r_{0}(x)dx - \int_{0}^{t}\lambda(x)dx\right) \nonumber \\
   	   		&\quad\quad\times nE_{\mcH}\left[\int_{0}^{t}\int_{0}^{t-r}\exp(-\mcH v)g_{W}(v)dv\lambda(r)dr\right] \nonumber \\
   	   		&\quad\quad \times \left\{E_{\mcH}\left[\int_{0}^{t}\int_{0}^{t-r}\exp(-\mcH v)g_{W}(v)dv\lambda(r)dr + \int_{0}^{t}\exp(-\mcH(t-r))\bar{G}_{W}(t-r)\lambda(r)dr\right]\right\}^{n-1}dt \nonumber \\
   	   		&+ \sum_{n=1}^{\infty}\int_{0}^{\infty}\frac{1}{n!}\exp\left(-\int_{0}^{t}r_{0}(x)dx - \int_{0}^{t}\lambda(x)dx\right) \nonumber \\
   	   		&\quad\quad \times n(n-1)E_{\mcH}\left[\mcH\int_{0}^{t}\int_{0}^{t-r}\exp(-\mcH v)g_{W}(v)dv\lambda(r)dr\int_{0}^{t}\exp(-\mcH(t-r))\bar{G}_{W}(t-r)\lambda(r)dr\right]\nonumber \\
   	   		&\quad\quad\times \left\{E_{\mcH}\left[\int_{0}^{t}\int_{0}^{t-r}\exp(-\mcH v)g_{W}(v)dv\lambda(r)dr + \int_{0}^{t}\exp(-\mcH(t-r))\bar{G}_{W}(t-r)\lambda(r)dr\right]\right\}^{n-2}dt  \nonumber
   	   \end{align}
   	   Let 
   	   	$a(t) = \int_{0}^{t}\int_{0}^{t-r}\exp(-\mcH v)g_{W}(v)dv\lambda(r)dr$.
   	   Through a change of variables,
   	    \[a(t) = \int_{0}^{t}\exp(-\mcH v)g_{W}(v)m(t-v)dv\] 
   	    Denoting 
   	    $b(t) = \int_{0}^{t}\exp(-\mcH(t-r))\bar{G}_{W}(t-r)\lambda(r)dr$, 
   	    the above simplifies to
   	   \begin{align}
   	   E[M] &= \sum_{n=1}^{\infty}\int_{0}^{\infty}\frac{1}{(n-1)!}r_{0}(t)E_{\mcH}[a(t)](E_{\mcH}[a(t) + b(t)])^{n-1}\exp\left(-\int_{0}^{t}r_{0}(x)dx-\int_{0}^{t}\lambda(x)dx\right)dt \nonumber \\
   	   	&\quad + \sum_{n=2}^{\infty}\int_{0}^{\infty}\frac{1}{(n-2)!}E_{\mcH}[\mcH a(t)b(t)](E_{\mcH}[a(t)+b(t)])^{n-2}dt \nonumber \\
   	   	&= \int_{0}^{\infty}r_{0}(t)\exp\left(-\int_{0}^{t}r_{0}(x)-\int_{0}^{t}\lambda(x)dx\right)
   	   				E_{\mcH}[a(t)]\left(\sum_{n=1}^{\infty}\frac{1}{(n-1)!}(E_{\mcH}[a(t)+b(t)])^{n-1}\right)dt \nonumber \\
   	   			&\quad+ \int_{0}^{\infty}\exp\left(-\int_{0}^{t}r_{0}(x)-\int_{0}^{t}\lambda(x)dx\right)E_{\mcH}[\mcH a(t)b(t)]\left(\sum_{n=2}^{\infty}\frac{1}{(n-2)!}(E_{\mcH}[a(t)+b(t)])^{n-2}\right)dt \nonumber \\
          &= \int_{0}^{\infty}r_{0}(t)\exp\left(-\int_{0}^{t}r_{0}(x)-\int_{0}^{t}\lambda(x)dx\right)
   	   				E_{\mcH}[a(t)]\exp\left(E_{\mcH}[a(t)+b(t)]\right)dt \nonumber \\
   	   			&\quad+ \int_{0}^{\infty}\exp\left(-\int_{0}^{t}r_{0}(x)-\int_{0}^{t}\lambda(x)dx\right)E_{\mcH}[\mcH a(t)b(t)]\exp\left(E_{\mcH}[a(t)+b(t)]\right)dt \nonumber \\
   	   		&= \int_{0}^{\infty}\exp\left(-\int_{0}^{t}r_{0}(x)-\int_{0}^{t}\lambda(x)dx+E_{\mcH}[a(t)+b(t)]\right)\left[r_{0}(t)E_{\mcH}[a(t)] + E_{\mcH}[\mcH a(t)b(t)]\right]dt \nonumber
   	    \end{align}
    Finally, we see that 
   \[\psi  = \frac{1}{\int_{0}^{\infty}S_{Y}(t)dt + \nu}\left\{\int_{0}^{\infty}\exp\left(-\int_{0}^{t}r_{0}(x)-\int_{0}^{t}\lambda(x)dx+E_{\mcH}[a(t)+b(t)]\right)\left[r_{0}(t)E_{\mcH}[a(t)] + E_{\mcH}[\mcH a(t)b(t)]\right]dt \right\}\]
 \end{proof}

 \section{Conclusion}
   This generalization provides a more realistic model to any server under nonconstant traffic. By specifying the service life distribution, the intensity function, and the distribution of the stresses, one can model almost any traffic situation. Future work will include investigating correlated Poisson streams, a partitioned server, and sufficient conditions for the maximum of the efficiency $\psi$. 
   



                                  

\appendix

\section{}

\begin{lem}
Let $\{N(t)\}$ be a nonhomogeneous Poisson process describing the arrival of client requests to the web server with intensity $\lambda(t)$, and let $\{T_{j}\}_{j=1}^{N(t)}$ be the arrival times of the client requests. Then, given $N(t) = n$, the conditional joint distribution of is equal to the joint distribution of the order statistics $T'_{[1]},...,T'_{[n]}$, where $T'_{1},...,T'_{n}$ are i.i.d. with pdf $f(x) = \frac{\lambda(x)}{m(t)}$, where $m(t) = \int_{0}^{t}\lambda(s)ds$  Formally,	  
\[f_{T_{1},...,T_{n}|N(t)=n}(t_{1},...,t_{n}|n) =f_{T'_{[1]},...,T'_{[n]}}(t_{1},...,t_{n}) = \frac{n!}{m(t)}\prod_{i=1}^{n}\lambda(t_{i})\]
 for $0\leq t_{1} \leq \ldots \leq t_{n} \leq t$. 
 \label{order_stat}
\end{lem}

\begin{proof}
Let $N(t) =n$, and let $0\leq t_{1}\leq \ldots \leq t_{n}\leq t$. Let $h_{i}, i = 1,...,n$ be small enough such that $t_{i}+h_{i} < t_{i+1} \frall i=1,...,n-1$. Let $A_{i}$ be the event where the server sees exactly one arrival in  	$[t_{i}, t_{i}+h_{i}]$,
 and 
 	$B$ 
 be the event that no events arrive outside the set
    $U = [0,t_{1}] \cup [t_{1}, t_{1}+h_{1}] \cup [t_{2}, t_{2}+h_{2}] \cup ... \cup [t_{n}, t_{n}+h_{n}] \cup [t_{n}+h_{n},t]$.
 Then 
 \begin{align}
 P(t_{i} \leq T_{i} \leq t_{i} + h_{i}, i=1,...,n | N(t) = n) &= 
 	\frac{P(A_{1} \cap A_{2} \cap ...\cap A_{n} \cap B)}{P(N(t) = n)} \nonumber 
 \end{align}
 
 By definition, $P(N(t) = n) = \frac{e^{-m(t)}m(t)^{n}}{n!}$. For each $i=1,...,n$, 
 \begin{align}
  P(t_{i} \leq T_{i} \leq t_{i} + h_{i}) &= P(N(t_{i} + h_{i}) - N(t_{i}) = 1) \nonumber \\
  													&= e^{-(m(t_{i} + h_{i}) - m(t_{i}))}(m(t_{i} + h_{i}) - m(t_{i})) \nonumber 
  \end{align}
 The complement of the set $U$ can be broken into the disjoint intervals $[0,t_{1}], [t_{1} + h_{1}, t_{2}],...[t_{i} + h_{i}, t_{i+1}],...[t_{n} + h_{n},t]$. Then 
 \begin{align}
  P(B) &= P(N(t) - N(t_{n}+h_{n}) = 0) P(N(t_{1}) - N(0) = 0)\prod_{i=1}^{n-1}P(N(t_{i+1}) - N(t_{i} + h_{i}) = 0) \nonumber \\
   &= e^{-m(t_{1})}e^{-(m(t) - m(t_{n} + h_{n}))}\prod_{i=1}^{n-1}e^{-(m(t_{i+1})- m(t_{i} + h_{i}))} \nonumber \\
   &= exp\left(-\left[m(t) + \sum_{i=1}^{n}m(t_{i}) - \sum_{i=1}^{n}m(t_{i} + h_{i})\right]\right) \nonumber 
 \end{align}
 Now,
 \begin{align}
   P(t_{i} \leq T_{i} \leq t_{i} + h_{i}, i=1,...,n &| N(t) = n) = \prod_{i=1}^{n}P(t_{i} \leq T_{i} \leq t_{i}+h_{i}) \nonumber \\
    &= \dfrac{\prod_{i=1}^{n}[m(t_{i} + h_{i}) - m(t_{i})]e^{-\sum_{i=1}^{n}[m(t_{i} + h_{i}) - m(t_{i})]}e^{-\left[m(t) + \sum_{i=1}^{n}m(t_{i}) - \sum_{i=1}^{n}m(t_{i} + h_{i})\right]}}{\frac{e^{-m(t)}m(t)^{n}}{n!} } \nonumber \\
   	&= n!\prod_{i=1}^{n}\frac{m(t_{i} + h_{i}) - m(t_{i})}{m(t)} \nonumber \\
   	&= n! \prod_{i=1}^{n}\frac{\lambda(t_{i})}{m(t)} \nonumber
  \end{align}
  by letting $h_{i} \to 0 \frall i$. 
\end{proof}

\begin{lem}
 \begin{equation}
 E\left[E\left[\exp\left(-\eta_{j}^{i_{j'}}\min(W_{j}, t-T_{j}')| W_{j} \right)\right]\right] = \frac{1}{m(t)}\left(m(t) - \eta_{j'}^{i_{j'}}\int_{0}^{t}\exp(-\eta_{j'}^{i_{j'}}w)m(t-w)\bar{G}_{W}(w)dw\right) 
 \end{equation}
 \label{cond_on_Wj}
\end{lem}

\begin{proof}
Two cases are considered: (1): $w \leq t$, and (2): $w > t$. For $w \leq t$,
 \begin{align}
  E\left[\exp\left(-\eta_{i_{j}}\min(W_{j}, t-T_{j}')\right)|W_{j} = w\right]
  	&= \int_{0}^{t-w}\exp(-\eta_{i_{j}}w)\frac{\lambda(x)}{m(t)}dx + \int_{t-w}^{t}\exp(-\eta_{i_{j}}(t-x))\frac{\lambda(x)}{m(t)}dx \nonumber \\
  	&= \exp(-\eta_{i_{j}}w)\frac{m(t-w)}{m(t)} + \exp(-\eta_{i_{j}}t)\int_{0}^{t}\exp(\eta_{i_{j}}x)\frac{\lambda(x)}{m(t)}dx \nonumber
  \end{align}
  For 
  	$ w>t$,
  \[E\left[\exp\left(-\eta_{i_{j}}\min(W_{j}, t-T_{j}')\right)|W_{j} = w\right] = \exp(-\eta_{i_{j}}t)\int_{0}^{t}\exp(\eta_{i_{j}}x)\frac{\lambda(x)}{m(t)}dx\]
  Therefore,
  \begin{align}
  E_{W}\left[\exp\left(-\eta_{i_{j}}\min(W_{j}, t-T_{j}')\right)\right] &= E_{W}\left[E\left[\exp\left(-\eta_{i_{j}}\min(W_{j}, t-T_{j}')\right)|W_{j} = w\right] \right] \nonumber \\
 	 &= \frac{1}{m(t)}\left(\int_{0}^{t}\exp(-\eta_{i_{j}}w)m(t-w)g_{W}(w)dw\right. \nonumber \\
  	&\quad +\underbrace{\exp(-\eta_{i_{j}}t)\int_{0}^{t}\int_{t-w}^{t}\exp(\eta_{i_{j}}x)\lambda(x)dxg_{W}(w)dw} \nonumber \\
  	&\quad + \bar{G}_{W}(t)\exp(-\eta_{i_{j}}t)\left.\int_{0}^{t}\exp(\eta_{i_{j}}x)\lambda(x)dx\right) \nonumber \\
  \intertext{Focusing on the second term, we make the change of variables $w = t-x$ and change the order of intergration, yielding a new second term.}
  &=  \frac{1}{m(t)}(\int_{0}^{t}\exp(-\eta_{i_{j}}w)m(t-w)g_{W}(w)dw \nonumber \\
  	&\quad +\underbrace{\exp(-\eta_{i_{j}}t)\int_{0}^{t}\exp(\eta_{i_{j}}x)\lambda(x)\int_{t-x}^{t}g_{W}(w)dwdx} \nonumber \\
  	&\quad + \bar{G}_{W}(t)\exp(-\eta_{i_{j}}t)\int_{0}^{t}\exp(\eta_{i_{j}}x)\lambda(x)dx) \nonumber \\
  \intertext{Combining the second and third terms:}
  &= \frac{1}{m(t)}\left(\int_{0}^{t}\exp(-\eta_{i_{j}}w)m(t-w)g_{W}(w)dw\right. \nonumber \\
  &\quad + \left.\exp(-\eta_{i_{j}}t)\int_{0}^{t}exp(\eta_{i_{j}}x)\lambda(x)\bar{G}_{w}(t-x)dx\right) \nonumber \\
  \intertext{Changing variables again in the second term, using $w=t-x$, we get}
  &= \frac{1}{m(t)}\left(\int_{0}^{t}\exp(-\eta_{i_{j}}w)m(t-w)g_{W}(w)dw\right. \nonumber \\
  	&\quad + \left.\int_{0}^{t}exp(-\eta_{i_{j}}w)\lambda(t-w)\bar{G}_{w}(w)dw\right) \nonumber \\
  \intertext{Integrating the first term by parts, we get}
  &= \frac{1}{m(t)}\left( \left.\left[-\exp(-\eta_{i_{j}}w)m(t-w)\bar{G}_{w}(w)\right]\right|_{0}^{t}\right. \nonumber \\
  	&\quad -\int_{0}^{t}(\eta_{i_{j}}\exp(-\eta_{i_{j}}w)m(t-w) + \exp(-\eta w)\lambda(t-w))\bar{G}_{W}(w)dw \nonumber \\
  	&\quad + \left.\int_{0}^{t}exp(-\eta_{i_{j}}w)\lambda(t-w)\bar{G}_{w}(w)dw\right) \nonumber \\
   &= \frac{1}{m(t)}\left(m(t) - \eta_{i_{j}}\int_{0}^{t}\exp(-\eta_{i_{j}}w)m(t-w)\bar{G}_{W}(w)dw\right) \nonumber
  \end{align}
\end{proof}

\begin{lem}
Under all conditions of Theorem~\ref{efficiency theorem}, the conditional distribution 	$f_{Y|\mfR, \mfV, \mfH, N}(t| \mfr, \mfv, \mfh, n)$ 
  is given by 
  	\[f_{Y|\mfR, \mfV, \mfH, N}(t| \mfr, \mfv, \mfh, n) = \exp \left(-\int_{0}^{t}r_{0}(s)ds - \sum_{j=1}^{n}\eta_{i_{j}}\min(v_{j}, t-r_{j})\right) \left(r_{0}(t) + \sum_{j=1}^{n}\eta_{i_{j}}\mathds{1}_{v_{j} > t-r_{j}}\right)\] 
  	\label{cond_dist_for_efficiency}
\end{lem}

\begin{proof}
Let 
  	$C = \{\mfR = \mfr, \mfV = \mfv, \mfH = \mfh, N(t) = n\}$. 
  Then
  
  \begin{equation}
  f_{Y|C}(t|c) = \lim\limits_{\Delta t \to 0}\frac{1}{\Delta t}\left(P(Y > t|C=c) - P(Y> t+ \Delta t| C=c)\right)
  \label{f_y|c}
  \end{equation}
  Recall $\mcB(s) = r_{0}(s) + \sum_{j=1}^{N(t)}\mathbb{I}(R_{j}\leq s \leq R_{j}+V_{j})$. It remains to derive $P(Y> t+ \Delta t| C=c)$. Let 
     		$A_{k} = \{N(t+ \Delta t) - N(t) = k\}$ 
     	be the event such that 
     		$k$ 
     	requests arrived between 
     		$t$ and $t+ \Delta t$. From the definition of a nonhomogenous Poisson process (\cite{ross}),
     		\begin{align}
     		  P(A_{0}) &= 1-\lambda(t)\Delta t + o(\Delta t) = e^{-(m(t+\Delta t) - m(t))} \nonumber \\
     		  P(A_{1}) &= \lambda(t)\Delta t + o(\Delta t) = (m(t+\Delta t) - m(t))e^{-(m(t+\Delta t) - m(t))}\nonumber \\
     					 &\vdots \nonumber \\
     		  P(A_{k}) &= \frac{(m(t+\Delta t) - m(t))^{k}}{k!}e^{-(m(t+\Delta t) - m(t))}\nonumber 
     		  \end{align}
     		  for 
     		  	$k \geq 2$. 
     		  By the Law of Total Probability, 
     		  \begin{equation}
     		  P(Y> t+ \Delta t | C=c) = \sum_{k=0}^{\infty}P(Y > t+ \Delta t|C \cap A_{k})P(A_{k})
     		  \label{P(Y>tdt_givenC)}
     		  \end{equation}
    ~\eqref{P(Y>tdt_givenC)} will be computed for each fixed $k$, and it will be shown that for $k \geq 2$, ~\eqref{P(Y>tdt_givenC)}  is of order $o(\Delta k)$. For $k=0$,
     \begin{equation}
        P(Y > t + \Delta t|C \cap A_{0})P(A_{0}) = e^{-\int_{0}^{t+\Delta t}\mcB(s)ds}(1-\lambda(t)\Delta t + o(\Delta t))  = e^{-\int_{0}^{t}\mcB(s)ds}e^{-\int_{t}^{t+\Delta t}\mcB(s)ds}(1-\lambda(t)\Delta t + o(\Delta t)) 
        \label{P(y|C_and_A0)P(A0)}
        \end{equation}
     where 
     \begin{align}
        	\int_{t}^{t+\Delta t}\mcB(s)ds &= \int_{t}^{t+\Delta t}r_{0}(s)ds + \left(\sum_{j=1}^{n}\eta_{i_{j}}\mathds{1}_{v_{j} > t-r_{j}}\right)\Delta t \nonumber \\
        		&= r_{0}(t) + o(\Delta t)+\left(\sum_{j=1}^{n}\eta_{i_{j}}\mathds{1}_{v_{j} > t-r_{j}}\right)\Delta t \nonumber 
      \end{align}
 Using the fact that $e^{-x + o(x)} = 1-x+o(x)$ for small $x$,
 \begin{align}
 	e^{-\int_{t}^{t+\Delta t}\mcB(s)ds} &= 1-\left(r_{0}(t) + \sum_{j=1}^{n}\eta_{i_{j}}\mathds{1}_{v_{j} > t-r_{j}}\right)\Delta t + o(\Delta t) \nonumber \\
    	\intertext{Substituting into ~\eqref{P(y|C_and_A0)P(A0)}, }
    	P(Y > t+ \Delta t | C \cap A_{0})P(A_{0}) &= \left[e^{-\int_{0}^{t}\mcB(s)ds} 1-\left(r_{0}(t) + \sum_{j=1}^{n}\eta_{i_{j}}\mathds{1}_{v_{j} > t-r_{j}}\right)\Delta t + o(\Delta t)\right]\left[1-\lambda(t)\Delta t + o(\Delta t)\right] \nonumber \\
    	&= e^{-\int_{0}^{t}\mcB(s)ds}\left(1- [r_{0}(t) + \sum_{j=1}^{n}\eta_{i_{j}}\mathds{1}_{v_{j} > t-r_{j}}]\Delta t -\lambda(t)\Delta t + o(\Delta t)\right)
    	\label{P(y|C_and_A0)P(A0)_updated}
 \end{align} 
 for $k=1$, one request has arrived in $[t, t+\Delta t]$ in addition to the $n$ fixed arrivals given by $C$. Call this arrival $t_{1}$, with corresponding time to completion $w_{1}$. We then have two cases:
    	\begin{enumerate}[(1)]
    	\item $t_{1} + w_{1} < t + \Delta t$, that is, the request arrives and is serviced before $t+ \Delta t$, or
    	\item $t_{1} + w_{1} > t+ \Delta t$. In this case, the service time is greater than $\Delta t$. 
    	\end{enumerate}
    	With both of these cases, the failure rate increases by an $\eta_{i_{t_{1}}}$ for a time smaller than $\Delta t$, which we will denote as $\Delta^{1}(t)$. Then
    	\[\Delta^{1}(t) = \begin{cases}
 							   	w_{1}, & t_{1} + w_{1} < t + \Delta t \\
 							   	t+\Delta t - t_{1}, & t_{1} + w_{1} > t + \Delta t
    	\end{cases}\]    
    Applying similar logic from earlier to calculate $P(Y > t+\Delta t | C \cap A_{1})P(A_{1})$,
    	\[\exp\left(-\int_{t}^{t+\Delta t}\mcB(s)ds\right) = \exp\left(-r_{0}(t) + [\sum_{j=1}^{n}\eta_{i_{j}}\mathds{1}_{v_{j} > t-r_{j}}]\Delta t + \eta_{t_{1}}^{i_{t_{1}}}\Delta^{1}t + o(\Delta t)\right) \]
   which simplifies to
   \[\exp\left(-\int_{t}^{t+\Delta t}\mcB(s)ds\right) = 1 - [r_{0}(t) + \sum_{j=1}^{n}\eta_{i_{j}}\mathds{1}_{v_{j} > t-r_{j}}]\Delta t + \eta_{t_{1}}^{i_{t_{1}}}\Delta^{1}t + o(\Delta t)\]
      	Now, 
      		\begin{align}
      	   	P(Y > t+ \Delta t|C \cap A_{1})P(A_{1}) &= e^{-\int_{0}^{t}\mcB(s)ds}\left([1 - [r_{0}(t) + \sum_{j=1}^{n}\eta_{i_{j}}\mathds{1}_{v_{j} > t-r_{j}}]\Delta t + \eta_{t_{1}}^{i_{t_{1}}}\Delta^{1}t + o(\Delta t)][\lambda(t)\Delta t +o(\Delta t)]\right) \nonumber \\
      	   		&= e^{-\int_{0}^{t}}\left(\lambda(t)\Delta t + o(\Delta t)\right) 
      	   	\label{P(y|C_and_A1)P(A1)}
      	   	\end{align}
    	For $k\geq 2$, it will be shown that the contribution to ~\eqref{P(Y>tdt_givenC)} are negligible.
    	
    		\[P(Y > t+\Delta t|C \cap A_{k}) = e^{-\int_{0}^{t}\mcB(s)ds}\left[1-[(r_{0}(t) + \sum_{j=1}^{n}\eta_{i_{j}}\mathds{1}_{v_{j} > t-r_{j}})\Delta t + \sum_{l=1}^{k}\eta_{t_{l}}^{i_{t_{l}}}\Delta^{l}t] + o(\Delta t)\right]\]
    Hence $\frall k \geq 2, P(Y > t+\Delta t | C \cap A_{k})\leq e^{-\int_{0}^{t}\mcB(s)ds}$
       	and thus
       	\begin{align}
       	   	\sum_{k = 2}^{\infty}P(Y > t+\Delta t | C \cap A_{k})P(A_{k}) &\leq e^{-\int_{0}^{t}\mcB(s)ds}\sum_{k=2}^{\infty}P(A_{k}) \nonumber \\
       	   		&= e^{-\int_{0}^{t}\mcB(s)ds}(1- P(A_{0}) - P(A_{1})) \nonumber \\
       	   		&= e^{-\int_{0}^{t}\mcB(s)ds}o(\Delta t) \nonumber
       	   	\end{align}
    Therefore,
       	\[P(Y > t+\Delta t | C) =  e^{-\int_{0}^{t}\mcB(s)ds}\left(1- (r_{0}(t) + \sum_{j=1}^{n}\eta_{i_{j}}\mathds{1}_{v_{j} > t-r_{j}})\Delta t \right) + o(\Delta t)\]
      and 
      \begin{align}
         	\frac{1}{\Delta t}(P(Y > t|C=c) - P(Y>t+ \Delta t | C=c))&= e^{-\int_{0}^{1}\mcB(s)ds}\left(r_{0}(t) + \sum_{j=1}^{n}\eta_{i_{j}}\mathds{1}_{v_{j} > t-r_{j}} + \frac{o(\Delta t)}{\Delta t}\right) \nonumber 
        \end{align}   
 Then by letting $\Delta t \to 0$,
 \[f_{Y|\mfR, \mfV, \mfH, N}(t| \mfr, \mfv, \mfh, n) = \exp \left(-\int_{0}^{t}r_{0}(s)ds - \sum_{j=1}^{n}\eta_{i_{j}}\min(v_{j}, t-r_{j})\right) \left(r_{0}(t) + \sum_{j=1}^{n}\eta_{i_{j}}\mathds{1}_{v_{j} > t-r_{j}}\right)\] 	
\end{proof}

\begin{lem}
\[f_{\mfR, \mfV, N(t)}(\mfr, \mfv, n) = \frac{1}{n!}\exp\left(-\int)_{0}^{t}\lambda(s)ds\right)\prod_{j=1}^{n}\lambda(r_{j})g_{W}(v_{j})\]
  \label{f(r,v,n)}
\end{lem}

\begin{proof}
  Let $\mcW = (W_{1},...,W_{n})$ be the service times under the condition that $N(t) = n$, and let $\mcW = \mcw = (w_{1},...,w_{n})$.  Let $(r_{[1]},...,r_{[n]})$ be the ordered vector of $\mfr$.
  There are $n!$ possible orderings of $\mfr$. Now, 
  $P(r_{[i]} \leq R_{[i]} \leq r_{[i]} + h_{i})$ 
  for some 
  $h_{i} > 0, i = 1,...,n$ 
  is given by 
  \[P(r_{[i]} \leq R_{[i]} \leq r_{[i]} + h_{i}) = e^{-(m(r_{[i]} + h_{i}) - m(r_{[i]})}(m(r_{[i]} + h_{i}) - m(r_{[i]}))\]
  In a similar manner to the proof of Lemma~\ref{order_stat}, we may see that the joint distribution of $\mfR$ is identical to the distribution of the order statistics of $\mfR$:
  \[f_{\mfR,N(t)}(\mfr,n) = \frac{1}{n!}\prod_{j=1}^{n}\lambda(r_{j})\exp\left(-\int_{0}^{t}\lambda(s)ds\right)\]
  $\mfR, \mfV$ are mutually independent. Therefore, 
   \[f_{\mfR, \mfV, N(t)}(\mfr, \mfv, n) = \frac{1}{n!}\exp\left(-\int)_{0}^{t}\lambda(s)ds\right)\prod_{j=1}^{n}\lambda(r_{j})g_{W}(v_{j})\]
  
  \end{proof}
\acks
My sincerest thanks to Andrzej Korzeniowski and Jason Hathcock for technical assistance and advice. 

%
%
%
%

\end{document}